\documentclass{amsart}

\usepackage{hyperref}
\usepackage[lite]{amsrefs}
\usepackage{amssymb}

\newtheorem{theorem}{Theorem}

\newcommand{\hatzero}{\hat{0}}
\newcommand{\myempty}{\varnothing}

\begin{document}

\title{An extension of Birkhoff's representation theorem to
infinite distributive lattices}
\author{Dale R. Worley}
\email{worley@alum.mit.edu}
\date{13 March 2023}

\begin{abstract}
Birkhoff's representation theorem for 
finite distributive lattices states that any finite distributive lattice
is isomorphic to the lattice of order ideals (lower sets) of the
partial order of
the join-irreducible elements of the lattice.
This theorem can be extended as follows:
A {\it non-finite} distributive lattice that is {\it locally finite} and
{\it has a $\hatzero$} is isomorphic to the lattice of {\it finite}
order ideals of the partial order of
the join-irreducible elements of the lattice.
In addition, certain {\it ``well ordering''} properties are shown
to be equivalent to the premises of the extended theorem.
\end{abstract}

\maketitle

\textit{Note added to the 2\textsuperscript{nd} version:}  Since the
publication of the
1\textsuperscript{st} version, I have been informed that the main result of this
paper (\ref{th:extension}) was previously published as
\cite{Stan2012}*{Prop. 3.4.3}.

Birkhoff's representation theorem for finite distributive lattices has
long been known:

\begin{theorem}[Birkhoff \cite{Birk1937}\cite{Birk1967}*{ch.\ III sec.\ 3 p.\ 58}]
Any finite distributive lattice $L$ is isomorphic to the
lattice of order ideals of the partial order of the join-irreducible
elements of $L$.
\end{theorem}

\begin{proof}[Proof sketch]
The proof is to demonstrate that
there is a one-to-one order-preserving correspondence between
elements of $L$ and order ideals of the partial order.
The order ideal
corresponding to an element $x \in L$ is simply the set of
join-irreducible elements of $L$ that are $\le x$, and
the element of $L$ corresponding to an order ideal $S$ of join-irreducible
elements is $\bigvee S$.
\end{proof}

(Note that in all of this discussion, it is required that $\hatzero$
in $L$ corresponds to the empty order ideal of the partial order.
If $\hatzero$ is seen as join-irreducible, then both $\myempty$ and
$\{\hatzero\}$ are order ideals, which does not match the structure of $L$.
So we define that an element $x$ is
``join-irreducible'' if  it cannot be expressed as the join of a {\it
finite} set of smaller elements; thus $\hatzero$ is the join of {\it
zero} elements and not join-irreducible.)

The motivation for this extension is that in the combinatorics of
tableaux, the distributive lattice of tableaux is formally defined as
the lattice of {\it finite} order ideals of an {\it infinite} base
poset.
The archetype example is that the {\it Young lattice} of {\it Ferrers
diagrams} equiv.\footnote{Two equivalent terms or conditions are
conjoined with ``equiv.''} {\it (unshifted) shapes} is the finite
order ideals of
the quadrant of the plane, $\mathbb{P}^2$.\footnote{I use $\mathbb{P}$
for $\{1, 2, 3, \ldots\}$ and $\mathbb{N}$ for $\{0, 1, 2, \ldots\}$.}
Another example is that the lattice of {\it shifted shapes} is the finite
order ideals of the octant of the plane, $\{(x,y): x,y \in \mathbb{P}
\text{ and } x \ge y\}$.
In these cases, the poset generating the lattice is already known, of
course, but it's clear that the construction of a lattice from a poset
is a general pattern.
(See e.g. \cite{Fom1994}*{sec.\ 2 and the $J(\cdot)$ operator}.)

The extension theorem shows that this construction is universal.

\begin{theorem}\label{th:extension}
If a distributive lattice $L$ is locally finite and has a $\hatzero$,
then $L$ is isomorphic to the lattice of finite order ideals of the
partial order
of the join-irreducible elements of $L$.
\end{theorem}

\begin{proof}

The proof is to demonstrate that there
is a one-to-one order-preserving correspondence between
elements of $L$ and finite order ideals of the partial order.
Composing these two maps in either order gives the identity map on the
appropriate domain, showing that they are the desired correspondence
and its inverse.

The order ideal
corresponding to an element $x \in L$ is the set of
join-irreducible elements of $L$ that are $\le x$, and
the element of $L$ corresponding to a finite order ideal $S$ of join-irreducible
elements is $\bigvee S$.

For any element $x \in L$, let $S$ be the
join-irreducible elements in $L$ that are $\le x$, which is an order
ideal of the join-irreducible elements.
Clearly the mapping from $x$ to $S$ is order-preserving.
$S$ is finite because it is a subset of the interval $[\hatzero, x]$.

Now let $y$ be the join of $S$.
We need to show $y = x$.
As a join of elements $\le x$, $y \le x$.
Since $[\hatzero, x]$ is finite, any decomposition of $x$
into joins must terminate, so $x$ is the join of some finite set $I$
of join-irreducibles in $[\hatzero, x]$.
$I \subset S$, so $x = \bigvee I \le \bigvee S = y$.

For any finite order ideal $S$ of join-irreducible elements of $L$, let $x$
be $\bigvee S$.
Clearly this mapping is order-preserving.
Let $T$ be the set of join-irreducible
elements $\le x$.  We need to show $S = T$.

Every element of $S$ clearly belongs to $T$.
Conversely, any element of $T$ is a join-irreducible element $\le x$
and must (by the join-primality of the elements of the finite set $S$,
which is implied by join-irreducibility in
distributive lattices) be $\le$ one of the members
of $S$, and therefore must itself (by the assumption that
$S$ is an order ideal) belong to $S$.

\end{proof}

While determining the antecedents of Theorem~\ref{th:extension}, I
found that certain ``well ordering'' properties are alternative
antecedents:

\begin{theorem}
For distributive lattice $L$, the conditions
\begin{enumerate}
  \item
$L$ is well-partially-ordered equiv.
$L$ contains no infinite descending chains equiv.
there does not exist $a_1, a_2, a_3, ... \in L$ such that
$a_1 > a_2 > a_3 > ...$, and
\item
$L$ contains no bounded infinite ascending chains equiv.
there does not exist $a_1, a_2, a_3, ..., b \in L$ such that
$a_1 < a_2 < a_3 < ... < b$,
\end{enumerate}
are equivalent to the conditions that $L$ is locally finite and has a
$\hatzero$.
\end{theorem}

\begin{proof}
Suppose $L$ satisfies the well-ordering conditions (1) and (2).
Then
it must have a minimal element, as otherwise an infinite descending chain could
be constructed.  Since $L$ is a lattice, there can be only one
minimal element, hence that element is $\hatzero$.

Let $I$ be the set of elements of $S = [x, y]$ that are
join-irreducible within $S$.
Every element of $S$ can be decomposed into
a finite join of elements of $I$.
If $S$ is infinite, then $I$ must be infinite
to allow constructing an
infinite number of different joins of elements of $I$.
Then by the infinite Ramsey theorem, $I$ would contain
either an infinite chain or an infinite antichain.

An infinite chain must be well-ordered (because $L$ contains no infinite
descending chains), and any infinite well-ordered set contains an initial subset
that is order-isomorphic to $\mathbb{N}$.
This subset would be an
infinite ascending chain, all of whose elements are $< y$, which
contradicts (2).

An infinite antichain would contain a countably infinite subset
$i_1, i_2, i_3, ...$ of incomparable elements of $I$.
Construct $i_1, i_1 \vee i_2, i_1 \vee i_2 \vee i_3, ...$\ .
Each of these is strictly greater than the previous, because if
$$\bigvee_{j=1}^n i_j = \bigvee_{j=1}^{n-1} i_j$$ then
$$i_n \vee \bigvee_{j=1}^{n-1} i_j = \bigvee_{j=1}^{n-1} i_j$$
$$i_n \le \bigvee_{j=1}^{n-1} i_j$$
$$i_n \le i_j \text{ for some } j < n \text{, because }
\bigvee_{j=1}^{n-1} i_j
\text{ is a finite join of join-irreducibles}$$
contradicting that the $i_j$ form an antichain.
Thus the $\bigvee_{j=1}^n i_j$ form an infinite ascending chain, all
of which are $< x$, which is assumed not to exist.
Thus, $S$ is finite.

Conversely, suppose $L$ is locally finite and has a $\hatzero$.
There can be no infinite
descending chain $a_1 > a_2 > a_3 > ... \in L$, because it would be
contained in $[\hatzero, a_1]$, which is finite.
Similarly, there can be no bounded infinite ascending chain 
$a_1 < a_2 < a_3 < ... < b \in L$, because it would be
contained in $[a_1, b]$, which is finite.
\end{proof}

\end{document}